\documentclass[11pt,a4paper]{article}

\usepackage{amsmath}
\usepackage{mathtools}
\usepackage{amsfonts}
\usepackage[mathscr]{eucal}
\usepackage{amsthm}
\usepackage{accents}
\usepackage{amssymb}
\usepackage{enumerate}
\usepackage{cite}

\DeclareMathOperator{\supp}{supp}

\newcommand{\meas}{m}

\usepackage{pgf,tikz}\usetikzlibrary{matrix, calc, arrows}
\usepackage[bookmarks=true, bookmarksopen=true, bookmarksopenlevel=4]{hyperref}  
\definecolor{myblue}{rgb}{0,0,0.6}     
\hypersetup{pdftitle={A note on properties of the restriction operator for Sobolev spaces on Rn},
            pdfauthor={Hewett, Moiola},
     colorlinks=true, linkcolor=myblue,  citecolor=myblue, filecolor=myblue,   urlcolor=myblue,  }
\allowdisplaybreaks[4]
\begin{document}
\newcommand{\rf}[1]{(\ref{#1})}
\newcommand{\ri}{{\mathrm{i}}}
\newcommand{\re}{{\mathrm{e}}}
\newcommand{\rd}{\mathrm{d}}
\newcommand{\R}{\mathbb{R}}
\newcommand{\N}{\mathbb{N}}
\newcommand{\C}{\mathbb{C}}
\newcommand{\cI}{\mathcal{I}}
\newcommand{\scrD}{\mathscr{D}}
\newcommand{\scrS}{\mathscr{S}}
\newcommand{\bd}{\mathbf{d}}
\newcommand{\bx}{\mathbf{x}}
\newcommand{\bxi}{\boldsymbol{\xi}}
\newcommand{\hsnorm}[1]{||#1||_{H^{s}(\bs{R})}}
\newcommand{\hnorm}[1]{||#1||_{\tilde{H}^{-1/2}((0,1))}}
\newcommand{\norm}[2]{\left\|#1\right\|_{#2}}
\newcommand{\normt}[2]{\|#1\|_{#2}}
\newtheorem{thm}{Theorem}[section]
\newtheorem{lem}[thm]{Lemma}
\newtheorem{defn}[thm]{Definition}
\newtheorem{prop}[thm]{Proposition}
\newtheorem{rem}[thm]{Remark}
\newcommand{\tH}{\widetilde{H}}
\newcommand{\Hze}{H_{\rm ze}} 	
\newcommand{\uze}{u_{\rm ze}}		
\DeclareRobustCommand
{\mathringbig}[1]{\accentset{\smash{\raisebox{-0.1ex}{$\scriptstyle\circ$}}}{#1}\rule{0pt}{2.3ex}}
\newcommand{\cirH}{\mathringbig{H}}
\newcommand{\cirHs}{\mathringbig{H}{}^s}
\newcommand{\cirHm}{\mathringbig{H}{}^m}
\newcommand{\deO}{{\partial\Omega}}
\newcommand{\OO}{{(\Omega)}}
\newcommand{\Rn}{{(\R^n)}}

\title{A note on properties of the restriction operator 
on Sobolev spaces}
\author{
D.\ P.\ Hewett\footnotemark[1], A.\ Moiola\footnotemark[2]}

\renewcommand{\thefootnote}{\fnsymbol{footnote}}
\footnotetext[1]
{Department of Mathematics, University College London, Gower Street, London WC1E 6BT, UK. Email: \texttt{d.hewett@ucl.ac.uk}}
\footnotetext[2]{Department of Mathematics and Statistics, University of Reading, Whiteknights PO Box 220, Reading RG6 6AX, UK.
Email: 
\texttt{a.moiola@reading.ac.uk}}

\maketitle
\renewcommand{\thefootnote}{\arabic{footnote}}

\begin{abstract}
In our companion paper \cite{ChaHewMoi:13} we studied a number of different Sobolev spaces on a general (non-Lipschitz) open subset $\Omega$ of $\R^n$, defined as closed subspaces of the classical Bessel potential spaces $H^s\Rn$ for $s\in\R$.
These spaces are mapped by the restriction operator to certain spaces of distributions on $\Omega$.
In this note we make some observations about the relation between these spaces of global and local distributions.
In particular, we study conditions under which the restriction operator is or is not injective, surjective and isometric between given pairs of spaces. 
We also provide an explicit formula for minimal norm extension (an inverse of the restriction operator in appropriate spaces) in a special case.
\end{abstract}
%

\section{Preliminaries}
We study properties of Sobolev spaces on a general (non-Lipschitz) open set $\Omega\subset\R^n$. In our companion paper \cite{ChaHewMoi:13} we studied two types of spaces: those consisting of distributions on $\R^n$ (specifically, $\tH^s(\Omega)$, $\cirH^s(\Omega)$, $H^s_{\overline\Omega}$, defined below), and those consisting of distributions on $\Omega$ itself (specifically, $H^s(\Omega)$ and $H^s_0(\Omega)$, again defined below).  
In this note we study properties of the restriction operator as a mapping between the two types of spaces. 
The results presented here, while elementary, do not seem to be available in the literature, which generally focusses on the more standard Lipschitz case (cf.\ e.g.~\cite{McLean}).
As in \cite{ChaHewMoi:13}, our motivation is the study of integral equations on non-Lipschitz sets. (For a concrete example see \cite[\S4]{ChaHewMoi:13}, where we consider boundary integral equation reformulations of wave scattering problems involving fractal screens.)

We begin by defining the Sobolev spaces involved. 
Our presentation  follows that of \cite{ChaHewMoi:13}, which in turn is broadly based on \cite{McLean}. 
Given $n\in \N$, let $\scrD(\R^n)$ denote the space of compactly supported smooth test functions on $\R^n$, and for any open set $\Omega\subset \R^n$ let
$\scrD(\Omega):=\{u\in\scrD(\R^n):\supp{u}\subset\Omega\}$. 
For $\Omega\subset \R^n$ let $\scrD^*(\Omega)$ denote the space of distributions on $\Omega$ (anti-linear continuous functionals on $\mathscr{D}(\Omega)$)\footnote{Following Kato \cite{Ka:95}, we work with dual spaces of anti-linear functionals, which simplifies certain aspects of the presentation. Our results translate trivially to dual spaces of linear functionals; see \cite[\S2]{ChaHewMoi:13} for further discussion.}.
Let $\mathscr{S}(\R^n)$ denote the Schwartz space of rapidly decaying smooth test functions on $\R^n$, and  $\mathscr{S}^*(\R^n)$ the dual space of tempered distributions (anti-linear continuous functionals on $\mathscr{S}(\R^n)$).
For $u\in \mathscr{S}(\R^n)$ we define the Fourier transform 
$\hat{u}(\bxi):= \frac{1}{(2\pi)^{n/2}}\int_{\R^n}\re^{-\ri \bxi\cdot \bx}u(\bx)\,\rd \bx , \;\; \bxi\in\R^n$, 
extending this definition to $\scrS^*(\R^n)$ in the usual way.
We define the Sobolev space $H^s(\R^n)$
by
\begin{align*}
H^s(\R^n):
= \big\{u\in \mathscr{S}^*(\R^n): \|u\|_{H^s\Rn}<\infty\big\},
\quad\text{where}\quad
\norm{u}{H^{s}(\R^n)}^2 = \int_{\R^n}(1+|\bxi|^2)^{s}|\hat{u}(\bxi)|^2\,\rd \bxi,
\end{align*}
which is a Hilbert space with the inner product
$\left(u,v\right)_{H^{s}(\R^n)} = \int_{\R^n}(1+|\bxi|^2)^{s}\,\hat{u}(\bxi)\overline{\hat{v}(\bxi)}\,\rd \bxi$.
For any $-\infty<s<t<\infty$, $H^t(\R^n)$ is continuously embedded in $H^s(\R^n)$ with dense image
and $\|u\|_{H^s(\R^n)}<\|u\|_{H^t(\R^n)}$ for all $0\ne u\in H^t(\R^n)$.
Recalling that for a multi-index $\boldsymbol\alpha\in\N_0^n$ we have ${\cal F}(\partial^{\boldsymbol\alpha}u/\partial \bx^{\boldsymbol\alpha})(\bxi)=(-\ri\bxi)^{\boldsymbol\alpha}\hat u(\bxi)$
and $|\boldsymbol\alpha|:=\sum_{j=1}^n\alpha_j$,
by Plancherel's theorem it holds  that
\begin{align}\label{eq:IntegralNorm}
(u,v)_{H^m\Rn}
=\sum_{\substack{\boldsymbol\alpha\in\N_0^n,\\ |\boldsymbol\alpha|\le m}}
\binom{m}{|\boldsymbol\alpha|}\binom{|\boldsymbol\alpha|}{\boldsymbol\alpha}
\int_{\R^n}\frac{\partial^{|\boldsymbol\alpha|} u}{\partial \mathbf x^{\boldsymbol\alpha}}(\bx)
\frac{\partial^{|\boldsymbol\alpha|} \overline v}{\partial \mathbf x^{\boldsymbol\alpha}}(\bx)
\rd \bx,\qquad m\in\N_0.
\end{align}
Hence functions with disjoint support are orthogonal in $H^m\Rn$ for $m\in\N_0$. 
But we emphasize that this is not in general true in $H^s \Rn$ for $s\in \R\setminus\N_0$.

For a closed $F\subset\R^n$, we define $H_F^s :=\big\{u\in H^s(\R^n): \supp(u) \subset F\big\}$. 
The question of whether a given set $E\subset\R^n$ can support nontrivial elements of $H^s\Rn$ will be important in what follows. 
This question was investigated in detail in \cite{HewMoi:15}, where we introduced the concept of $s$-nullity. (This concept is referred to by some authors as $(2,-s)$-polarity, see e.g.\ \cite[\S13.2]{Maz'ya}.)
\begin{defn}
For $s\in\R$ we say that a set $E\subset\R^n$ is $s$-null if there are no non-zero elements of $H^{s}(\R^n)$ supported entirely inside $E$ (equivalently, if $H^{s}_{F}=\{0\}$ for every closed set $F\subset E$).
\end{defn}

There are many different ways to define Sobolev spaces on an open subset $\Omega\subset \R^n$. In \cite{ChaHewMoi:13} we studied the following three spaces, all of which are closed subspaces of $H^s(\R^n)$, hence Hilbert spaces with respect to the inner product inherited from $H^s(\R^n)$:
\begin{align*}
 H_{\overline{\Omega}}^s &:= \big\{u\in H^s(\R^n): \supp(u) \subset \overline{\Omega}\big\},
 &&s\in\R,\\
\tH^s(\Omega)&:=\overline{\scrD(\Omega)}^{H^s(\R^n)},&&s\in\R,\\
\cirHs(\Omega) &:= \big\{u\in H^s(\R^n): u= 0 \mbox{ a.e. in } \Omega^c\big\}
= \big\{u\in H^s(\R^n): \meas\big(\supp{u}\cap(\Omega^c)\big) = 0\big\}, && s\ge0;
\end{align*}
here $m(\cdot)$ denotes Lebesgue measure on $\R^n$.
%
%
(We note that $\cirHs(\Omega)$ can also be identified with the set of functions defined on $\Omega$ which can be extended by zero to produce functions of the same Sobolev regularity on the whole of $\R^n$, see Remark \ref{rem:Hsze}.)
These three spaces satisfy the inclusions 
\begin{align*}
\tH^s(\Omega)\subset \cirHs(\Omega)\subset H^s_{\overline\Omega}
\end{align*}
(with $\cirHs(\Omega)$ present only for $s\geq0$). 
If $\Omega$ is sufficiently smooth (e.g.\ $C^0$) then the three sets coincide, but in general all three can be different (this issue is studied in \cite[\S3.5]{ChaHewMoi:13}).



Another way to define Sobolev spaces on an open set $\Omega$ is by restriction from $H^s(\R^n)$. For $s\in\R$ let
$$
H^s(\Omega):=\big\{u\in \scrD^*(\Omega): u=U|_\Omega \textrm{ for some }U\in H^s(\R^n)\big\},
\quad
\|u\|_{H^{s}(\Omega)}:= \min_{\substack{U\in H^s(\R^n)\\ U|_{\Omega}=u}}\normt{U}{H^{s}(\R^n)},
$$
where $U|_\Omega$ denotes the restriction of the distribution $U$ to $\Omega$ in the standard sense (cf.~\cite[p.~66]{McLean}). 
The inner product on $H^s\OO$ can be written as
$(u,v)_{H^s(\Omega)} : = (Q_sU,Q_sV)_{H^s(\R^n)}$,
for $u,v\in H^s(\R^n)$, where $U,V\in H^s(\R^n)$ are such that $U|_\Omega = u$, $V|_\Omega = v$ and
$Q_s$ is the orthogonal projection $Q_s:H^s\Rn\to(H^s_{\Omega^c})^\perp$, see \cite[\S3.1.4]{ChaHewMoi:13}.
It follows that the restriction operator
\begin{equation}\label{eq:RestrIsUnitary}
|_\Omega :(H^s_{\Omega^c})^\perp\to H^s(\Omega) \qquad \text{is a unitary isomorphism.}
\end{equation}
%
We also introduce the closed subspace of $H^s(\Omega)$ defined by
\begin{equation*}
H^s_0(\Omega):=\overline{\scrD(\Omega)\big|_\Omega}^{H^s(\Omega)}.
\end{equation*}
The question of when $H^s\OO$ and $H^s_0\OO$ are equal is studied in detail in \cite[\S3.6]{ChaHewMoi:13}.

For any open $\Omega\subset\R^n$, closed $F\subset\R^n$ and $s\in\R$,
the following dual space realisations hold, in the sense of unitary isomorphism (see \cite[\S 3.2]{ChaHewMoi:13}):
\begin{align}
\label{eq:Duality}
\big(H^s\OO\big)^*&=\tH^{-s}\OO,&
\quad
\big(\tH^s\OO\big)^*&=H^{-s}\OO,
\\
(H^s_F)^*&=\big(\tH^{-s}(F^c)\big)^\bot,&
\quad
\big(H^s_0\OO\big)^*&=
(\tH^{-s}(\Omega) \cap H^{-s}_{\partial\Omega})^{\perp, \tH^{-s}(\Omega)}.
\nonumber
\end{align}
The duality pairings corresponding to these realisations are defined in terms of the duality pairing $\langle u,v\rangle_{H^s\Rn\times H^{-s}\Rn}
=\int_{\R^n} \hat u(\bxi)\overline{\hat v(\bxi)}\rd\bxi$, which extends the $L^2\Rn$ scalar product.


\section{Properties of the restriction operator \texorpdfstring{$|_\Omega: H^s(\R^n)\to H^s(\Omega)$}{}}
\label{subsec:restriction}


In this section we examine the relationship between the spaces $\tH^s(\Omega),\cirHs(\Omega),H^s_{\overline\Omega}\subset H^s(\R^n)$, whose elements are distributions on $\R^n$, and the spaces $H^s(\Omega)$ and $H^s_0(\Omega)$, whose elements are distributions on $\Omega$. The two types of space are linked by the restriction operator $|_\Omega:H^s(\R^n)\to H^s(\Omega)$, 
and in this section we investigate some of its properties. 
In particular we ask: for a given value of $s$ and an appropriate pair of subspaces $X\subset H^s(\R^n)$, $Y\subset H^s(\Omega)$, when is $|_\Omega:X\to Y $ (i) injective? (ii) surjective? (iii) a unitary isomorphism?

We start by recalling that
$|_\Omega:X\to H^s(\Omega)$ is continuous with norm at most one, for any subspace $X\subset H^s(\R^n)$, 
and that $|_\Omega:(H^s_{\Omega^c})^\perp\to H^s(\Omega)$ is a unitary isomorphism.

For Lipschitz $\Omega$ with bounded $\deO$ we have the following result, which states that $|_\Omega:\tH^s(\Omega)\to H^s_0(\Omega)$ is an isomorphism for certain values of $s$.
(As in \cite{ChaHewMoi:13,McLean}, we say that $\Omega$ is Lipschitz if its boundary can be locally represented as the graph, suitably rotated, of a Lipschitz function from $\R^{n-1}$ to $\R$, with $\Omega$ lying only on one side of $\partial\Omega$.)
The result for $s\geq0$ is classical (see e.g.\ \cite[Theorem 3.33]{McLean}); the extension to $-1/2<s<0$ is proved below (it is an immediate consequence of Lemma~\ref{lem:restriction}\rf{a5} and 
\cite[Corollary~3.29(ix)]{ChaHewMoi:13}).
In interpreting this result one should recall that for Lipschitz $\Omega$ 
it holds that $H^s_0(\Omega)=H^s(\Omega)$ if and only if $s\leq 1/2$ \cite[Corollary~3.29(ix)]{ChaHewMoi:13}
and also that $\tH^s(\Omega)=\cirHs(\Omega)=H^s_{\overline{\Omega}}$ for all $s\in\R$ 
(see \cite[Lemma 3.15]{ChaHewMoi:13}, which follows from \cite[Theorems 3.29 and 3.21]{McLean}),
with $\cirHs(\Omega)$ present only for $s\geq 0$.
\begin{lem}
\label{lem:restrictionLipschitz}
If $\Omega$ is Lipschitz, $\deO$ is bounded, and $s> -1/2$, $s\not\in \{1/2,3/2,...\}$, then $|_\Omega:\tH^s(\Omega) \to H^s_0(\Omega)$ is an isomorphism (with norm at most one).
\end{lem}

We would like to understand to what extent this result generalises to non-Lipschitz $\Omega$, and also how $|_\Omega$ acts on the spaces $\cirHs(\Omega)$ and $H^s_{\overline{\Omega}}$ in the case where these are not equal to $\tH^s(\Omega)$. Some partial results in this direction are provided by the following lemma.
\begin{lem} \label{lem:same}
Let $\Omega\subset \R^n$ be open and let $s\in\R$. Then:
\begin{enumerate}[(i)]
\item \label{b4}
$|_\Omega:H^s_{\overline \Omega}\to H^s(\Omega)$ is injective if and only if $\partial\Omega$ is $s$-null.
\item \label{e4}
For $s\geq 0$, $|_\Omega:\cirHs(\Omega) \to H^s(\Omega)$ is injective; if $s\in\N_0$ then it is a unitary isomorphism onto its image in $H^s(\Omega)$.
%
%
%
\item \label{f4}
For $s\geq 0$, $|_\Omega:\tH^s(\Omega)\to H^s_0(\Omega)$ is injective and has dense image; if $s\in\N_0$ then it is a unitary isomorphism onto $H^s_0(\Omega)$.
\end{enumerate}
\end{lem}
\begin{proof}
Part \rf{b4} is obvious from the definition of the restriction operator. 
For part \rf{e4}, the injectivity statement is obvious, since if $u\in \cirHs(\Omega)$ and $u|_\Omega=0$ then $u=0$. 
That $|_\Omega:\cirHs(\Omega) \to H^s(\Omega)$ is a unitary isomorphism onto its image in $H^s(\Omega)$ when $s\in\N_0$ follows because in this case the $H^s(\R^n)$ inner product (see~\eqref{eq:IntegralNorm}) can be written as a sum of integrals over products of functions/derivatives in the ``physical'' space (as opposed to Fourier space), so disjoint support is enough to guarantee orthogonality.
Hence when $s\in\N_0$ we have $\cirHs(\Omega)\subset (H^s_{\Omega^c})^\perp$, and we know by \eqref{eq:RestrIsUnitary} that $|_\Omega$ is a unitary isomorphism from $(H^s_{\Omega^c})^\perp$ onto $H^s(\Omega)$.
Part \rf{f4} follows from part \rf{e4} and the density of $\scrD(\Omega)$ in both $\tH^s(\Omega)$ and $H^s_0(\Omega)$ (since the image of a closed set under an isometry is closed).
\end{proof}

\begin{rem}
By combining Lemma \ref{lem:same}\rf{b4} with the results concerning $s$-nullity in \cite[Lemma~3.10]{ChaHewMoi:13} (see also \cite{HewMoi:15})
one can derive a number of corollaries. 
For example: 
\emph{(i)} For every open $\Omega$ there exists $-n/2\le s_\Omega\le n/2$ such that $|_\Omega:H^s_{\overline \Omega}\to H^s(\Omega)$ is always injective for $s>s_\Omega$ and never injective for $s<s_\Omega$.
In particular, $|_\Omega:H^s_{\overline \Omega}\to H^s(\Omega)$ is always injective for $s>n/2$ and never injective for $s<-n/2$. 
\emph{(ii)} If $\Omega$ is Lipschitz (even with unbounded boundary), then $|_\Omega:\tH^s(\Omega)=H^s_{\overline \Omega}\to H^s(\Omega)$ is injective if and only if $s\ge -1/2$. 
\emph{(iii)} For every $-1/2\leq s_*\leq 0$ there exists a $C^0$ open set $\Omega$ for which $|_\Omega:\tH^s(\Omega)=H^s_{\overline \Omega}\to H^s(\Omega)$ is injective for all $s>s_*$ and not injective for all $s<s_*$. 
\end{rem}

\begin{rem}\label{rem:Hsze}
To expand on Lemma~\ref{lem:same}\rf{e4}, we note that the restriction operator $|_\Omega:\cirHs(\Omega) \to \{u\in H^s(\Omega): \uze\in H^s(\R^n)\}\subset H^s(\Omega)$ is a bijection, where we denote by $\uze$ the extension of $u\in H^s\OO$ from $\Omega$ to $\R^n$ by zero and $u\mapsto \uze$ is the inverse of $|_\Omega$, see also \cite[Remark~3.1]{ChaHewMoi:13}.
\end{rem}

\begin{rem}
Lemma~\ref{lem:same}\rf{f4} and 
\cite[Remark~3.32]{ChaHewMoi:13}, which follows from \cite[Chapter 1, Theorem~11.7]{LiMaI},
imply that if $\Omega$ is $C^\infty$ and bounded, and if $s\in\{1/2,3/2,\ldots\}$, then the restriction $|_\Omega:\tH^s\OO\to H^s_0\OO$ is not surjective, demonstrating the sharpness of Lemma~\ref{lem:restrictionLipschitz}.
%
\end{rem}


\begin{rem}
If $\Omega$ is a Lipschitz open set with bounded boundary, Lemma~\ref{lem:restrictionLipschitz} and the definition of $H^s\OO$ give that 
$|_\Omega:\tH^s\OO\cap H^1\Rn\to H^s_0\OO\cap H^1\OO$ is an isomorphism for all $0\le s\le1$, $s\ne 1/2$.
Then \cite[Theorem~3.40]{McLean} gives that 
\begin{align*}
&|_\Omega:\tH^s\OO\cap H^1\Rn\to H^1\OO && \text{is an isomorphism for }0\le s<1/2, \\
&|_\Omega:\tH^s\OO\cap H^1\Rn\to H^1_0\OO && \text{is an isomorphism for }1/2<s\le1.
\end{align*}
Characterising $(\tH^{1/2}\OO\cap H^1\Rn)|_\Omega$, and deriving similar results for non-Lipschitz sets, appear to be open problems.
\end{rem}

Lemma \ref{lem:same}\rf{e4} and \rf{f4} only deal with the case $s\geq 0$. 
In the next lemma we relate properties of the restriction operator acting on $\tH^s(\Omega)$ for $s$ and $-s$. In particular, this lemma allows us to infer the statement of Lemma \ref{lem:restrictionLipschitz} for $-1/2<s<0$ from the classical statement for $0<s<1/2$ (recalling that $H^s_0(\Omega)=H^s(\Omega)$ for Lipschitz $\Omega$ and $s\leq 1/2$). 
For clarity, in this lemma and its proof we denote the restriction operator acting on $\tH^s(\Omega)$ as 
$|^s_\Omega:\tH^s(\Omega)\to H^s(\Omega)$. 
The proof of the lemma makes use of the fact that we can characterise the Banach space adjoint of ${|^s_\Omega}$ in terms of $|^{-s}_\Omega$, using the dual space realisations  \eqref{eq:Duality}.
\begin{lem}
\label{lem:restriction}
Let $\Omega\subset \R^n$ be non-empty and open, and let $s\in\R$. Then
\begin{enumerate}[(i)]
\item \label{a5}
$|^s_\Omega:\tH^s(\Omega)\to H^s(\Omega)$ is bijective if and only if $|^{-s}_\Omega:\tH^{-s}(\Omega)\to H^{-s}(\Omega)$ is bijective.
\item \label{b5} $|^{-s}_\Omega:\tH^{-s}(\Omega)\to H^{-s}(\Omega)$ is injective if and only if $|^{s}_\Omega:\tH^{s}(\Omega)\to H^{s}(\Omega)$ has dense image; i.e.\ if and only if $H^s_0(\Omega)=H^s(\Omega)$.
\end{enumerate}
\end{lem}
\begin{proof}
Let $\cI_s:H^{-s}(\Omega)\to(\tH^s(\Omega))^*$ and $\cI_{s}^*:\tH^{s}(\Omega)\to(H^{-s}(\Omega))^*$ be the unitary isomorphisms defined in \cite[eq.~(21)]{ChaHewMoi:13}:
\begin{align*}
\cI_s u (v)= \langle U,v \rangle_{s}
\;\; \mbox{ and } \;\;
\cI_s^*v(u) = \langle v,U \rangle_{-s}, 
\quad \mbox{ for } u\in H^{-s}(\Omega), \,v\in\tH^s(\Omega),
\end{align*}
where $U\in H^{-s}(\R^n)$ denotes \textit{any} extension of $u$ with $U|_\Omega=u$.
Let ${|^s_\Omega}^*:(H^s(\Omega))^*\to(\tH^s(\Omega))^*$ denote the Banach space adjoint (i.e.\ the transpose) of $|^s_\Omega$, defined by $({|^s_\Omega}^* \;l)(\phi) =  l(\phi |^s_\Omega)$ for $l\in (H^s(\Omega))^*$ and $\phi\in \tH^s(\Omega)$. 
We can characterise ${|^s_\Omega}^*$ in terms of $|^{-s}_\Omega$, using $\cI_s$ and $\cI_{-s}^*$. Precisely, it holds that
${|^s_\Omega}^*\; \cI_{-s}^* = \cI_s |^{-s}_\Omega$.
To see this, simply note that, by the definition of $\cI_{s}$ and $\cI_{-s}^*$,
\begin{align*}
({|^s_\Omega}^* \; \cI_{-s}^* u)(v) = (\cI_{-s}^* u)(v|^s_\Omega) = \langle u,v \rangle_s = \big(\cI_s (u|^{-s}_\Omega)\big)(v),
\quad 
u\in\tH^{-s}(\Omega),\;v\in\tH^{s}(\Omega).
\end{align*}
From this characterisation the statements of the lemma follow immediately using classical functional-analytic results, e.g.\ \cite[Corollary 2.18 and Theorem 2.20]{Brezis}.
\end{proof}

We have seen that the restriction operator $|_\Omega:\tH^s\OO\to H^s_0\OO$ is an isomorphism for Lipschitz $\Omega$ and for $s> -1/2$, $s\not\in \{1/2,3/2,...\}$. 
The next proposition shows that this result also extends to the case where $\Omega$ is a finite union of disjoint Lipschitz open sets, even when the union is not itself Lipschitz. Note that we do not require the closures of the constituent open sets to be mutually disjoint. 
The result therefore applies, for example, to the prefractal sets generating the Sierpinski triangle 
\cite[Figure~4(a)]{ChaHewMoi:13},
which are finite unions of equilateral triangles touching at vertices.
\begin{prop}
\label{prop:restrictionUnionLipschitz}
The statements of Lemma \ref{lem:restrictionLipschitz} extend to finite disjoint unions of Lipschitz open sets with bounded boundaries. 
\end{prop}
\begin{proof}
The injectivity statement follows from the $s$-nullity of finite unions of Lipschitz boundaries for $s\ge-1/2$ (cf.\ \cite[Lemma~3.10 (v) and (xix)]{ChaHewMoi:13}).
Surjectivity follows from Lemma \ref{lem:Iso} below.
\end{proof}
\begin{lem}
\label{lem:Iso}
Let the open set $\Omega\subset\R^n$ be the union of the disjoint open sets $\{A_j\}_{j=1}^N$, $N\in\N$, and suppose that the restrictions $|_{A_j}:\tH^s(A_j)\to H^s_0(A_j)$ are surjective for all $1\le j\le N$. 
Then also $|_\Omega:\tH^s\OO\to H^s_0\OO$ is surjective.
\end{lem}
\begin{proof}
In this proof we denote by $|_{\Omega_1,\Omega_2}$ the restriction operator from $\scrD^*(\Omega_1)$ to $\scrD^*(\Omega_2)$, whenever $\Omega_2\subset\Omega_1\subset\R^n$ are open sets. 
Fix $u\in H^s_0\OO$. 
Then, for all $1\le j\le N$, $u|_{\Omega,A_j}\in H^s(A_j)$ belongs to $H^s_0(A_j)$ since $\Omega$ is a disjoint union and so $(\scrD(\Omega)|_{\R^n,\Omega})|_{\Omega,A_j}=\scrD(A_j)|_{\R^n,A_j}$. By assumption, $u|_{\Omega,A_j}=w_j|_{\R^n,A_j}$ for some $w_j\in \tH^s(A_j)\subset \tH^s\OO$.
Finally $w:=\sum_{j=1}^N w_j\in \tH^s\OO$ satisfies $w|_{\R^n,\Omega}=u$
(using the fact that any test function $\phi\in\scrD(\Omega)$ can be uniquely decomposed as a sum $\phi=\sum_{j=1}^N \phi_j$ where $\phi_j\in\scrD(A_j)$), 
and this shows that $u$ is in the range of $|_{\R^n,\Omega}$, as required.
\end{proof}

For $s\ge0$ we can rephrase the results of this section as follows.
For any open set $\Omega$, the restriction $|_\Omega:\tH^s(\Omega)\to H^s_0(\Omega)$ is continuous with norm one, is injective, has dense image, and the zero extension $u\mapsto \uze$ is its right inverse on its image, i.e.\ $\uze|_\Omega=u$ for all $u\in \tH^s\OO|_\Omega$.
Furthermore, for $s\ge0$, the following conditions are equivalent:
\begin{enumerate}[(i)]
\item $|_\Omega:\tH^s(\Omega)\to H^s_0(\Omega)$ is an isomorphism;
\item $|_\Omega:\tH^s(\Omega)\to H^s_0(\Omega)$ is surjective;
\item the zero extension $u\mapsto \uze$ is continuous $H^s_0(\Omega)\to\tH^s(\Omega)$;
\item there exists $c>0$ such that for all $\phi\in\scrD\OO$ and $\Phi\in H^s\Rn$ such that $\Phi|_\Omega=\phi$  we have 
$\|\Phi\|_{H^s\Rn}\ge c\, \|\phi\|_{H^s\Rn}$.
\end{enumerate}
By Proposition \ref{prop:restrictionUnionLipschitz} we know all these conditions hold for disjoint unions of Lipschitz open sets with bounded boundary. 
But results about the surjectivity (or otherwise) of $|_\Omega:\tH^s(\Omega)\to H^s_0(\Omega)$ on more general $\Omega$ do not seem to be available in the literature and in this case we only know (by Lemma \ref{lem:same}\rf{f4}) that the conditions above are true for $s\in\N_0$.
The following therefore appear to be open questions: For which $\Omega$ 
are (i)--(iv) true  
for all $s>-1/2$, $s\not\in \{1/2,3/2,...\}$? 
For which values of $s$ are they satisfied for a given $\Omega$?

\subsection{When is  \texorpdfstring{$|_\Omega:\tH^s(\Omega)\to H^s_0(\Omega)$}{the restriction} a unitary isomorphism?}
\label{subsec:unitary}

To study when  $|_\Omega:\tH^s(\Omega)\to H^s_0(\Omega)$ is a unitary isomorphism, we first note the equivalences in the following lemma. 
We emphasize that the norm on the left-hand side of the equality in part \rf{s2} in the lemma is the minimal one among the $H^s(\R^n)$-norms of all the extensions of $\phi|_\Omega$, while that on the right-hand side uses $\phi=0$ in $\Omega^c$. 
\begin{lem}
\label{lem:normEquality}
Let $\Omega$ be a non-empty open subset of $\R^n$ and let $s\in\R$. The following are equivalent.
\begin{enumerate}[(i)]
\item \label{s1}
$|_\Omega:\tH^s(\Omega)\to H^s_0(\Omega)$ is a unitary isomorphism.
\item \label{s2}
$\big\|\phi|_\Omega\big\|_{H^s(\Omega)} = \|\phi\|_{H^s(\R^d)}$ for all $\phi\in \scrD(\Omega).$
\item \label{s3}
$\scrD(\Omega) \subset (H^s_{\Omega^c})^\perp.$
\end{enumerate}
\end{lem}
\begin{proof}
The implications \rf{s1}$\Rightarrow$\rf{s2} and \rf{s3}$\Rightarrow$\rf{s1} are trivial (the latter holding by the density of $\scrD(\Omega)$ in $\tH^s(\Omega)$ and \rf{eq:RestrIsUnitary}). \rf{s2}$\Rightarrow$\rf{s3} follows because $|_\Omega:(H^s_{\Omega^c})^\perp\to H^s(\Omega)$ is an isometry (cf.\ \eqref{eq:RestrIsUnitary}). Explicitly, if $\phi\in\scrD(\Omega)$ then $\phi=\phi_1+\phi_2$ for a unique pair $\phi_1\in (H^s_{\Omega^c})^\perp$ and $\phi_2\in H^s_{\Omega^c}$. It follows that $\|\phi\|_{H^s(\R^d)}= \|\phi_1\|_{H^s(\R^d)}+ \|\phi_2\|_{H^s(\R^d)}$, and that $\|\phi|_\Omega\|_{H^s(\Omega)}=\|\phi_1|_\Omega\|_{H^s(\Omega)}=\|\phi_1\|_{H^s(\R^d)}$. So if the equality in \rf{s2} holds we must have that $\phi_2=0$, i.e. $\phi\in(H^s_{\Omega^c})^\perp$.
\end{proof}

Lemma \ref{lem:normEquality} allows us to prove the following proposition, which shows that the unitarity property holds whenever the complement of $\Omega$ is negligible (in the sense of $s$-nullity). 
An extreme example is the punctured space $\Omega=\R^n\setminus\{\mathbf0\}$, for which the proposition holds for any $s\geq -n/2$.
\begin{prop}
\label{prop:isometry}
Let $s\in\R$, and let $\Omega$ be an open subset of $\R^n$ such that $\Omega^c$ is $s$-null. Then $|_\Omega:\tH^s(\Omega)\to H^s_0(\Omega)$ is a unitary isomorphism.
\end{prop}
\begin{proof}
The assumption that $\Omega^c$ is $s$-null means that $(H^s_{\Omega^c})^\perp=(\{0\})^\perp=H^s(\R^n)\supset \scrD(\Omega)$, so part \rf{s3} of Lemma \ref{lem:normEquality} holds and hence the result follows. 
\end{proof}

Conversely, we can demonstrate that, when the complement of $\Omega$ is not negligible, $|_\Omega:\tH^s(\Omega)\to H^s_0(\Omega)$ is not in general a unitary isomorphism except when $s\in\N_0$. 
\begin{prop}
\label{prop:IsometryIffNotN}
Assume that $\Omega$ is non-empty, open and bounded. 
Then the three equivalent statements in Lemma~\ref{lem:normEquality} hold  
if and only if $s$ is a non-negative integer.
\end{prop}
\begin{proof}
We have seen in Lemma \ref{lem:same}\rf{f4} that $|_\Omega:\tH^s(\Omega)\to H^s_0(\Omega)$ is a unitary isomorphism when $s\in\N_0$, for any $\Omega$. 
When $s\not\in \N_0$ and $\Omega$ is bounded we shall prove that this does not hold by showing that statement  \rf{s3} of Lemma \ref{lem:normEquality} fails. 
Take any $\phi\in \scrD(\Omega)$ and define the translate $\phi_{\bd}(\bx):=\phi(\bx-\bd)$ for $\bd\in\R^n$. 
Then $\phi_{\bd}\in \scrD(\R^n)$. 
In fact, since $\Omega$ is assumed bounded, for large enough $|\bd|$ we have that $\phi_{\bd}\in  \scrD(\overline\Omega^c)\subset H^s_{\Omega^c}$, so that in particular $\supp{\phi}\cap\supp{\phi_{\bd}}$ is empty. 
Define $\chi(\bd):=(\phi,\phi_{\bd})_{H^s(\R^n)}=(\phi(\cdot),\phi(\cdot-\bd))_{H^s(\R^n)}$. Then the formula for the Fourier transform of a translate gives
\begin{align*}
\chi(\bd)=(\phi,\phi_{\bd})_{H^s(\R^n)} =
\int_{\R^n} \re^{\ri \bd\cdot\bxi} \mu(\bxi)\,\rd \bxi,
\end{align*}
where $\mu(\bxi):=(1+|\bxi|^2)^s |\hat{\phi}(\bxi)|^2=(1+\xi_1^2+\ldots \xi_n^2)^s |\hat{\phi}(\bxi)|^2$, with $\bxi=(\xi_1,\ldots,\xi_n)$. 
Since $\mu(\bxi)$ is an element of $\scrS(\R^n)$, $\chi(\bd)$ is also an element of $\scrS(\R^n)$, with Fourier transform $\hat\chi(\bxi)=(2\pi)^{n/2}\mu(\bxi)$. But for $s\neq 0,1,2,3,...$ the function $\mu(\bxi)$ does not extend to an entire function on $\C^n$ because the factor $(1+\xi_1^2+\ldots \xi_n^2)^s$ has singularities in $\C^n$. (E.g.\ for $n=1$, these singularities occur at the points $\bxi=\pm\ri$). Hence by the Paley--Wiener Theorem (see e.g.\ \cite[Theorem 2.3.21]{Grafakos})
$\chi(\bd)$ cannot have compact support in $\R^n$. 
As a result we can always find $\bd$ with $|\bd|$ large enough that $\phi_{\bd}\in  \scrD(\overline\Omega^c)\subset H^s_{\Omega^c}$ and
$\chi(\bd)=(\phi,\phi_{\bd})_{H^s(\R^n)}\neq 0$.
\end{proof}

\begin{rem}
In proving the ``only if'' statement in Proposition \ref{prop:IsometryIffNotN} we required $\Omega$ to be bounded.
With minor modifications the same proof works for some unbounded $\Omega$.
A first example is when $\Omega^c$ is bounded with non-empty interior.
%
%
A second example is when either $\Omega$ itself or $\overline{\Omega}^c$, the interior of the complement of $\Omega$, assumed to be non-empty, is contained in the hypograph $\{\bx\in\R^n,x_n>g(x_1,\ldots,x_{n-1})\}$, where $g:\R^{n-1}\to \R$ satisfies $\lim_{|\widetilde\bx|\to\infty}g(\widetilde\bx)=\infty$; the proof of Proposition~\ref{prop:IsometryIffNotN} works in this case because $\chi(-\bd)=\overline{\chi(\bd)}$.
%
The result does not hold for every open set $\Omega$, as Proposition~\ref{prop:isometry} demonstrates. However, we conjecture that the statement of Proposition~\ref{prop:IsometryIffNotN} holds for any $\Omega$ for which $\Omega^c$ has non-empty interior. But proving this conjecture appears to be an open problem. 
\end{rem}

\begin{rem}
\label{rem:extensions}
Proposition \ref{prop:IsometryIffNotN} illustrates the fact that Sobolev norms with non-natural-number indices are non-local. In particular it implies that given any $s\in\R\setminus\N_0$, any $\phi\in\scrD(\R^n)$ and any (arbitrarily large) bounded set $\Omega$ containing the support of $\phi$, there exists $\psi\in\scrD(\R^n)$ with support in $\Omega^c$ such that $\|\phi+\psi\|_{H^s(\R^n)}<\|\phi\|_{H^s(\R^n)}$. 

As an illustrative example, 
we exhibit a 
sequence $\{\Phi_N\}_{N\in\N_0}\subset H^{-1}(\R)$ of distributions with compact support $\supp\Phi_N\subset[0,2N]$ such that each one of them is an extension of the preceding one 
(i.e.\  $\Phi_{N+1}|_{(-\infty,2N+\frac12)}=\Phi_N|_{(-\infty,2N+\frac12)}$ for all $N\ge0$)
and their norms are strictly decreasing in $N$, i.e.\ 
$\|\Phi_{N+1}\|_{H^{-1}(\R)}<\|\Phi_{N}\|_{H^{-1}(\R)}$. 
Such a sequence can be defined as follows: choose any $0<\alpha<1/\re$ and set 
(where $\delta_x$ denotes the Dirac delta centred at $x\in\R$)\footnote{To fit our convention of using anti-linear functionals, $\delta_x$ acts on test functions $\phi\in\scrS(\R^n)$ by $\delta_x(\phi)=\overline{\phi(x)}$.}
$$\Phi_0:=\delta_0,\qquad 
\Phi_N:=\sum_{k=0}^{2N}(-\alpha)^k\delta_k
=\Phi_{N-1}-\alpha^{2N-1}\delta_{2N-1}+\alpha^{2N}\delta_{2N}, \quad N\in \N.$$
The Fourier transform formula 
$\hat\delta_x=\frac1{\sqrt{2\pi}}\re^{\ri x\xi}$ 
and the identity $\int_\R (1+\xi^2)^{-1}\re^{\ri a\xi}\,\rd\xi=\pi \re^{-|a|}$
imply that the $H^{-1}(\R)$-scalar product of two delta functions is\begin{equation}\label{eq:DeltaProd}
(\delta_x,\delta_y)_{H^{-1}(\R)}=\frac1{2} \re^{-|x-y|},
\end{equation}
giving
\begin{align*}
\|\Phi_N\|^2_{H^{-1}(\R)}
=\frac12\sum_{k=0}^{2N}\alpha^{2k} + \sum_{0\le j<k\le 2N} (-\alpha)^{j+k}\re^{-(k-j)}.
\end{align*}
With some manipulations it is not difficult to prove that every extension strictly reduces the norm:
\begin{align*}
&\|\Phi_N\|^2_{H^{-1}(\R)}-\|\Phi_{N-1}\|^2_{H^{-1}(\R)}
= -\frac{\alpha^{4N-2}}{2(1+\alpha \re)}\left((1+\alpha^2)(1-\alpha\re)+2(1-\alpha/\re)(\alpha \re)^{1-2N} \right)<0.
\end{align*}
We point out that while the sequence $\{\|\Phi_{N}\|_{H^{-1}(\R)}\}_{N=1}^\infty$ is decreasing, our results in \S\ref{subsec:minext} (equation \eqref{eqn:Qminus1} in particular) imply that for every $N\in\N_0$ the extension of $\Phi_N|_{(-\infty,2N+\frac12)}$ with minimal $H^{-1}(\R^n)$ norm is supported in $(-\infty,2N+\frac12]$ and has the expression  
$\Phi_N+c\delta_{2N+1/2}$ for some $c\in\C$.
\end{rem}


\section{The space \texorpdfstring{$(H^s_{\Omega^c})^\perp$}{HsOmegaCperp} and minimal norm extensions}
\label{subsec:minext}

From \eqref{eq:RestrIsUnitary}, elements of $(H^s_{\Omega^c})^\perp$ are the extensions of elements of $H^s(\Omega)$ with minimal $H^s(\R^n)$ norm. 
In this short section we make some remarks on the nature of elements of $(H^s_{\Omega^c})^\perp$, and on minimal norm extensions. 
We also refer the reader to the related discussion in Remark \ref{rem:extensions} above.

For $m\in \N_0$, the fact that functions with disjoint support are orthogonal in $H^m(\R^n)$
(cf.\ \rf{eq:IntegralNorm} and the sentence following it) implies that $ \cirHm(\Omega)\subset (H^m_{\Omega^c})^\perp$. Thus we have
\begin{align*}
\tH^m(\Omega)\subset \cirHm(\Omega)\subset 
(H^m_{\Omega^c})^\perp \subset \big(\tH^{m}(\overline{\Omega}^c)\big)^\perp , \qquad m\in \N_0,
\end{align*}
which by duality \eqref{eq:Duality} implies that
\begin{align}\label{eq:MinExtRelationDual}
(H^{-m}_{\Omega^c})^\perp 
\subset \tH^{-m}(\Omega) \subset H^{-m}_{\overline\Omega} , \qquad m\in \N_0.
\end{align}
In particular,
minimal extensions from $H^{-m}\OO$ to $H^{-m}(\R^n)$ are supported in $\overline\Omega$. 
Hence if $u\in H^{-m}(\Omega)$ then there exists $U\in H^{-m}_{\overline\Omega}$ with $U|_{\Omega}=u$;
furthermore given any such $U$ the minimal extension of $u$ is given by $Q_{-m} U=U + w$ where $w\in H^{-m}_{\partial \Omega}$.

For example, if $\Omega=(a,b)\subset\R$ then the action of $Q_{-m}$ on $U\in H^{-m}_{\overline\Omega}$ can be written explicitly, since $H^{-m}_{\partial \Omega}$ is finite-dimensional and its elements 
are (derivatives of) delta functions supported in $\partial\Omega=\{a,b\}$. 
In particular, for $U\in H^{-1}_{\overline\Omega}$,
\begin{align}
\label{eqn:Qminus1}
Q_{-1} U = U + c_a \delta_a + c_b\delta_b,
\qquad \textrm{for some }c_a,c_b\in\C.
\end{align}
Using \eqref{eq:DeltaProd}, minimisation of $\|U+c_a\delta_a+c_b\delta_b\|_{H^{-1}(\R)}^2$  shows that 
\begin{align*}
c_a=\frac{(U,\delta_b)_{H^{-1}(\R)}-\re^{b-a}(U,\delta_a)_{H^{-1}(\R)}}{\sinh (b-a)},
\quad
c_b=\frac{(U,\delta_a)_{H^{-1}(\R)}-\re^{b-a}(U,\delta_b)_{H^{-1}(\R)}}{\sinh (b-a)}.
\end{align*}
For instance, if $u\in H^{-1}(\Omega)$ is given by $u = \delta_x$ for some $a<x<b$ (viewed as a distribution on $\Omega=(a,b)$), then clearly $U:=\delta_x$ (viewed as a distribution on $\R$) is an extension of $u$, whose projection onto $(H^{-1}_{\Omega^c})^\perp$ is given by \rf{eqn:Qminus1}. 
In this case the choice of $c_a,c_b$ that minimises the $H^{-1}(\R^n)$ norm of \rf{eqn:Qminus1} is
\begin{align*}
c_a=-\frac{\sinh(b-x)}{\sinh(b-a)}, \qquad c_b=-\frac{\sinh (x-a)}{\sinh(b-a)},
\end{align*}
which give
$$\|\delta_x\|_{H^{-1}(\Omega)}^2=
\|\delta_x+c_a\delta_a+c_b\delta_b\|_{H^{-1}(\R)}^2
=\frac{\sinh(b-x)\sinh(x-a)}{\sinh(b-a)}
<\frac12=\|\delta_x\|_{H^{-1}(\R)}^2.
$$

However, in general $(H^s_{\Omega^c})^\perp\not\subset H^s_{\overline\Omega}$ when $-s\not\in\N_0$, i.e.\ elements of $(H^s_{\Omega^c})^\perp$ do not generally have their support in $\overline\Omega$. Explicit expressions for the minimal-norm extensions of elements of $H^1(\Omega)$ and $H^2(\Omega)$ for the special case $\Omega=(a,b)\subset\R$ have been presented in \cite[Lemma 4.12]{InterpolationCWHM} and lead to the formulas for the norms:
\begin{align*}
\|\phi\|_{H^1(\Omega)}^2 =&
|\phi(a)|^2 + |\phi(b)|^2 + \int_{a}^b\left(|\phi|^2+ |\phi^\prime|^2\right) \,\rd x,\\
\|\phi\|_{H^2(\Omega)}^2=
&|\phi(a)|^2+|\phi'(a)|^2+|\phi(a)-\phi'(a)|^2+
|\phi(b)|^2+|\phi'(b)|^2+|\phi(b)+\phi'(b)|^2
\\
&+\int_a^b (|\phi|^2+2|\phi'|^2+|\phi''|^2) \,\rd x.
\end{align*}
(Note that we have corrected a sign typo present in \cite[eq.~(26)]{InterpolationCWHM}.)


%
%

\end{document}